\documentclass[11pt]{amsart}  
\usepackage{amssymb, eucal}
\usepackage[all,arc]{xy}
\usepackage{enumerate}
\usepackage{color}

\newenvironment{tfae}{
\begin{enumerate}}{\end{enumerate}}

\newtheorem{prop}{Proposition}[section]
\newtheorem{lemma}[prop]{Lemma}
\newtheorem{theorem}[prop]{Theorem}

\theoremstyle{definition}
\newtheorem{definition}[prop]{Definition}

\newtheorem{examples}[prop]{Examples}
\newtheorem{remark}[prop]{Remark}

\counterwithin*{equation}{section}

\newdir{ >}{{}*!/-8pt/@{>}}

\def\mathrmdef#1{\expandafter\def\csname#1\endcsname{{\rm#1}}}
\mathrmdef{Aut}\mathrmdef{can}\mathrmdef{ev}\mathrmdef{Ker}\mathrmdef{id}\mathrmdef{Id}\mathrmdef{lex}\mathrmdef{max}\mathrmdef{op}\mathrmdef{pr}
\mathrmdef{prod}\mathrmdef{sym}\mathrmdef{Ult}\mathrmdef{Pt}\mathrmdef{SPt}\mathrmdef{End}\mathrmdef{cod}\mathrmdef{SplExt}

\def\mathsfdef#1{\expandafter\def\csname#1\endcsname{{\rm\sf#1}}}

\def\CC{\mathcal{C}}
\def\VVAut{V\hspace*{-1mm}\text{-}\Aut}
\def\VGrp{V\hspace*{-1mm}\text{-}\Grp}
\newcommand{\VCat}{V\hspace*{-1mm}\text{-}\Cat}

\mathsfdef{Alg} \mathsfdef{Bool} \mathsfdef{C}\mathsfdef{Cat} \mathsfdef{CLC}
\mathsfdef{Comp}\mathsfdef{Conn}\mathsfdef{CompHaus}\mathsfdef{ContLat}
\mathsfdef{Cont}\mathsfdef{D}\mathsfdef{Fam} \mathsfdef{Ps} \mathsfdef{Grp}
\mathsfdef{Haus}\mathsfdef{Inf}
\mathsfdef{Lat}\mathsfdef{LC}\mathsfdef{Mon} \mathsfdef{MultiOrd}
\mathsfdef{Ord}\mathsfdef{RelAlg}\mathsfdef{PsTop}\mathsfdef{Rel}
\mathsfdef{Set} \mathsfdef{SGrp} \mathsfdef{Sup} \mathsfdef{Top}
\mathsfdef{OrdGrp}\mathsfdef{ROrdGrp} \mathsfdef{OrdAb}
\mathsfdef{CMon} \mathsfdef{A} \mathsfdef{B}\mathsfdef{Met}\mathsfdef{MetGrp}\mathsfdef{UMet} \mathsfdef{UMetGrp}

\def\relto{{\longrightarrow\hspace*{-2.8ex}{\mapstochar}\hspace*{2.6ex}}}

\def\two{\mbox{\bf 2}}

\begin{document}  
\title{Algebraic exponentiation and action representability for $V$-groups}  
\author{Maria Manuel Clementino}
\address{CMUC, Department of Mathematics, University of Coimbra, 3000-143 Coimbra, Portugal}\thanks{}
\email{mmc@mat.uc.pt}

\author{Andrea Montoli}
\address{Dipartimento di Matematica ``Federigo Enriques'', Universit\`{a} degli Studi di Milano, Via Saldini 50, 20133 Milano, Italy}
\email{andrea.montoli@unimi.it}
\thanks{}

\begin{abstract}
We show that the category of $V$-groups, where $V$ is a cartesian quantale, so in particular the category of preordered groups, is locally algebraically cartesian closed with respect to the class of points underlying the product $V$-category structure. We obtain this by observing that such points correspond to ($\VCat)$-enriched functors from a $V$-group, seen as a one-object $V$-category, to the category $\VGrp$ of $V$-groups. Moreover, we show that the actions corresponding to points underlying the product $V$-category structure are representable.
\end{abstract}

\maketitle

\section{Introduction}
The notion of protomodular category \cite{Bourn protomod} allowed a common description of the homological properties of a large family of structures, including non-abelian algebraic structures like groups, rings, Lie algebras, as well as the duals of elementary toposes. In order to differentiate better the behaviors of such categories, some additional conditions have been considered in the literature. Among them, two particularly strong ones are action representability \cite{BJK} and local algebraic cartesian closedness \cite{Gray}. The first one describes categorically the fact that split epimorphisms between groups correspond to group actions, and those can be represented as homomorphisms to the group of automorphisms. The second one is the algebraic counterpart of the notion of local cartesian closedness, obtained by replacing the basic fibration with the so-called \emph{fibration of points}.

Recently it has been observed that several non-protomodular categories maintain most of the good properties of the protomodular ones, when the attention is restricted to a suitable class of points
. This led to the notion of $S$-protomodular category \cite{BMMS S-protomodular}, where $S$ is a pullback stable class of points. The guiding example of an $S$-protomodular category is the category $\Mon$ of monoids with the class $S$ of Schreier points \cite{MMS mon w op}. In \cite{CMFM} some categorical properties of the category $\OrdGrp$ of preordered groups have been explored. In particular, it was observed that $\OrdGrp$ is $S$-protomodular with respect to the class $S$ of points whose domain, seen as a semidirect product, is equipped with the product preorder. Some of the properties studied in \cite{CMFM} have been then extended in \cite{CM} to $V$-groups, where $V$ is a commutative and unital quantale. The case of preordered groups is an instance of this general situation, when $V = \two$ is the two-element order $0 \leq 1$.

Relative versions of the additional notions we mentioned before have been considered in $S$-protomodular categories \cite{Schreier book, MMS SLACC}. The aim of the present paper is to show that, when $V$ is a cartesian quantale (meaning that $V$ is a cartesian monoidal category with respect to the order in $V$), the category $\VGrp$ of $V$-groups is $S$-protomodular with respect to the class $S$ of points whose domains have the product structure. Moreover, it is both $S$-action representable and $S$-LACC with respect to the same class $S$. This is obtained by extending to the case of $V$-groups the analogous results proved for the category of monoids in \cite{Schreier book, MMS SLACC}.

\section{Protomodular and $S$-protomodular categories}
We start by recalling the definition of a protomodular category \cite{Bourn protomod}. There are several equivalent ways to express it. The one which is more suitable for our purposes makes use of some features of the so-called fibration of points. By \emph{point} we mean a split epimorphism with a chosen splitting. More explicitly, a \textit{point} in a category $\mathcal{C}$ is a pair $(f, s)$, where $f \colon A \to B$, $s \colon B \to A$ and $fs = 1_B$. A morphism between two points $(f, s)$ and $(f', s')$ is a pair $(g, h)$ of morphisms such that the following diagram commutes both upwards and downwards:
\[ \xymatrix{ A \ar[r]^g \ar@<-2pt>[d]_f & A' \ar@<-2pt>[d]_{f'} \\
B \ar[r]_h \ar@<-2pt>[u]_s & B'. \ar@<-2pt>[u]_{s'} } \]
If the category $\mathcal{C}$ has pullbacks, the functor
\[ \text{cod} \colon \Pt(\mathcal{C}) \to \mathcal{C} \]
with domain the category of points in $\mathcal{C}$, associating to any point its codomain, is a fibration, called the \emph{fibration of points}.

\begin{definition}[\cite{Bourn protomod}]
A category $\mathcal{C}$ with pullbacks is \emph{protomodular} if all change-of-base functors of the fibration of points are conservative.
\end{definition}

For pointed categories, protomodularity is equivalent to the validity of the Split Short Five Lemma. This notion captures some key properties of non-abelian algebraic structures, like groups, rings, Lie algebras, but it also includes several less algebraic examples, like the dual of every elementary topos. In order to distinguish all these examples in a categorical way, some additional conditions have been considered by different authors. One we are particularly interested in is the following:

\begin{definition}[\cite{Gray}] \label{def lacc category}
A finitely complete category $\mathcal{C}$ is \emph{locally algebraically cartesian closed}, or briefly LACC, if, for every morphism $h \colon X \to Y$, the corresponding change-of-base functor of the fibration of points $h^*\colon \Pt_Y(\CC)\to \Pt_X(\CC)$ has a right adjoint.
\end{definition}

This is a rather restrictive condition, which is not satisfied by many examples of protomodular categories. The main examples having this property are the category of groups and the one of Lie algebras over a fixed commutative ring. Some other categories satisfy this property for some good subfibrations of the fibration of points. In order to express this relative condition in a precise manner, we recall the following definition:

\begin{definition}[\cite{BMMS S-protomodular}] \label{def S-protomodular category}
Let $S$ be a class of points in a finitely complete category $\mathcal{C}$ which is stable under pullbacks along any morphism. The category $\mathcal{C}$ is \emph{$S$-protomodular} if:
\begin{enumerate}
\item every point $(f, s)$ in $S$ is a strong point, meaning that, in every pullback diagram
\[ \xymatrix{ P \ar[r]^{\bar{g}} \ar@<-2pt>[d]_{\bar{f}} & A \ar@<-2pt>[d]_f \\
C \ar[r]_g \ar@<-2pt>[u]_{\bar{s}} & B, \ar@<-2pt>[u]_{s} } \]
the morphisms $\bar{g}$ and $s$ are jointly extremal epimorphic;

\item $\SPt(\mathcal{C})$ is closed under finite limits in
$\Pt(\mathcal{C})$.
\end{enumerate}
\end{definition}

The class $S$ determines a subfibration of the fibration of points:
\[ S\text{-cod} \colon \SPt(\mathcal{C}) \to \mathcal{C}, \]
and, in every $S$-protomodular category, the change-of-base functors of this subfibration are conservative \cite[Theorem 3.2]{BMMS S-protomodular}. So, in particular, every protomodular category is $S$-protomodular with respect to the class $S$ of all points.

The guiding example of $S$-protomodular category is the category \Mon\ of monoids with respect to the class $S$ of \emph{Schreier} points:

\begin{definition}[\cite{MMS mon w op}] \label{def Schreier split epi}
A point $\xymatrix{ A \ar@<-2pt>[r]_f & B \ar@<-2pt>[l]_s }$ in the category \Mon\ of monoids is a \emph{Schreier point} if, for any $a \in A$, there exists a unique $\alpha$ in the kernel $\text{Ker}(f)$ of $f$ such that $a = \alpha \cdot sf(a)$.
\end{definition}

As shown in \cite{MMS mon w op}, Schreier points correspond to classical monoid actions, where an action of a monoid $B$ on a monoid $X$ is a monoid homomorphism $B \to \End(X)$, $\End(X)$ being the monoid of endomorphisms of $X$. Such correspondence is obtained via a semidirect product construction, generalizing the well-known fact that split epimorphisms of groups correspond to group actions. In the case of monoids, the correspondence is kept by restricting the study to Schreier split epimorphisms. It was observed in \cite{Schreier book} that the class $S$ of Schreier points in \Mon\ satisfies all the conditions to turn \Mon\ an $S$-protomodular category. Moreover, many algebraic and homological properties of the category \Grp\ of groups are still valid in $\Mon$, when restricted to Schreier points.

Another example of an $S$-protomodular category which is relevant for our purposes is the category \OrdGrp\ whose objects are the preordered groups, that are groups equipped with a preorder (i.e. a reflexive and transitive relation) such that the group operation is monotone, and whose morphisms are the monotone group homomorphisms. Some categorical properties of \OrdGrp\ have been explored in \cite{CMFM}. In particular, it was observed there that \OrdGrp\ possesses a good class of points turning it an $S$-protomodular category. Let us briefly describe them. A split extension (with its kernel) in \OrdGrp\
\[ \xymatrix{ X \ar[r]^k & A \ar@<-2pt>[r]_p & B \ar@<-2pt>[l]_s} \]
is such that, in $\Grp$, $A$ is isomorphic to a semidirect product $X \rtimes_\varphi B$, and the split extension is isomorphic to
\begin{equation} \label{eq:split}
\xymatrix{ X \ar[r]^-{\langle 1,0\rangle} & X \rtimes_\varphi B \ar@<-2pt>[r]_-{\pi_B} & B. \ar@<-2pt>[l]_-{\langle 0,1\rangle}}
\end{equation}
As observed in \cite{CMFM}, on the group $X \rtimes_\varphi B$ there may be many preorder structures turning \eqref{eq:split} a split extension in $\OrdGrp$. Let $S$ be the class of points \eqref{eq:split} in \OrdGrp\ such that the product projection $\pi_X \colon X \rtimes_\varphi B \to X$ is monotone (it is not a group homomorphism, unless the action $\varphi$ is trivial). This is equivalent to saying that the preorder on $X \rtimes_\varphi B$ is the product preorder or that, if we restrict the point to the positive cones, we obtain a Schreier point of monoids. It was shown in \cite{CMFM} that \OrdGrp\ is $S$-protomodular with respect to this class $S$. The same fact remains true if we replace \OrdGrp\ with the category of $V$-groups for a cartesian quantale $V$, as we will show in the next sections. \\

Properties like local algebraic cartesian closedness can be studied in the context of $S$-proto\-modular categories, by restricting the attention to the subfibration $S\text{-cod} \colon \SPt(\mathcal{C}) \to \mathcal{C}$ of the fibration of points:

\begin{definition}[\cite{MMS SLACC}] \label{def S-lacc category}
An $S$-protomodular category $\mathcal{C}$ is \emph{$S$-locally algebraically cartesian closed}, or \emph{$S$-LACC}, if, for every morphism $h \colon X \to Y$, the corresponding change-of-base functor of the subfibration $S\text{-cod}$:
\[ h^* \colon \SPt_Y(\mathcal{C}) \to \SPt_X(\mathcal{C}) \]
has a right adjoint.
\end{definition}

It was proved in \cite[Proposition 4.4]{MMS SLACC} that \Mon\ is $S$-LACC with respect to the class $S$ of Schreier points. One of our main aims will be to show that the same argument can be extended to preordered groups, and, more generally, to $V$-groups, at least when $V$ is a cartesian quantale. \\

Another additional condition which is often considered to classify protomodular categories is action representability:

\begin{definition}[\cite{BJK}] \label{def action representable cat}
A pointed protomodular category $\mathcal{C}$ is \emph{action representable} if, for every object $X$, the functor $\Pt(\;\; , X) \colon \mathcal{C} \to \Set$ sending $Y \in \mathcal{C}$ to the set of isomorphism classes of split extensions with codomain $Y$ and kernel $X$ is representable.
\end{definition}

Also this condition is quite restrictive. Again, the main examples are the category of groups, where the representing object is the group of automorphisms, and the one of Lie algebras, where the representing object is the Lie algebra of derivations. We can consider a relative version of this notion, with respect to a good class $S$ of points:

\begin{definition} \label{def S-action representable cat}
We say that a pointed $S$-protomodular category $\mathcal{C}$ is \emph{$S$-action representable} if, for every object $X$, the functor $\SPt(\;\; , X) \colon \mathcal{C} \to \Set$ sending $Y \in \mathcal{C}$ to the set of isomorphism classes of split extensions in $S$ with codomain $Y$ and kernel $X$ is representable.
\end{definition}

It was shown in \cite[Section 5]{Schreier book} that $\Mon$ is $S$-action representable for the class $S$ of Schreier points. We will show that the category of $V$-groups, where $V$ is a cartesian quantale, is $S$-action representable for a suitable class $S$. This will apply, in particular, to the case of preordered groups. (For the study of representability of other special classes of points of preordered groups we refer to \cite[Section 5]{CRuivo}.)

\section{$V$-categories and $V$-groups}
Let $V$ be a commutative and unital quantale. This means that $V$ is a complete lattice, with top element $\top$ and bottom $\bot$, equipped with a symmetric and associative tensor product $\otimes$, with unit $\kappa$, which preserves arbitrary joins. We will also assume that arbitrary joins distribute over finite meets.

A \emph{$V$-relation} from a set $X$ to a set $Y$ is a map $X\times Y\to V$. A \emph{$V$-category} is a pair $(X,a)$, where $X$ is a set and $a\colon X\relto X$ is a $V$-relation such that
\[1_X\leq a\mbox{ and }a\cdot a\leq a.
\]
The two conditions above can be equivalently formulated as:
\begin{description}
\item[\rm (R)] $(\forall \, x\in X)\;\;\;\kappa\leq a(x,x)$,
\item[\rm (T)] $(\forall \, x,x',x''\in X)\;\;\; a(x,x')\otimes a(x',x'')\leq a(x,x'')$.
\end{description}
Property (R) is called \emph{reflexivity} while (T) is \emph{transitivity}; they are also called, respectively, \emph{unit} and \emph{associativity} axioms.

A \emph{$V$-functor} $f\colon(X,a)\to(Y,b)$ is a map $f \colon X \to Y$ such that $f\cdot a\leq b\cdot f$, or, in other terms,
\[(\forall x,x'\in X)\;\;\;a(x,x')\leq b(f(x),f(x')).\]
$V$-categories and $V$-functors form a category, denoted by $\VCat$.

A \emph{$V$-group} is a $V$-category $(X,a)$ equipped with a group structure $(X,+\colon X\times X\to X,i\colon X\to X,0\colon I \to X)$ such that $+\colon(X,a)\otimes (X,a)\to (X,a)$ is a $V$-functor. We observe that $0\colon(1,\kappa)\to (X,a)$, as every map from $(1,\kappa)$ to $(X,a)$, is a $V$-functor, while we do not require that the inversion map $i\colon(X,a)\to(X,a)$ is a $V$-functor. We will use the additive notation although our groups need not be abelian. A \emph{$V$-homomorphism} $f\colon(X,a)\to(Y,b)$ between two $V$-groups is a $V$-functor which is also a group homomorphism. We will denote by $\VGrp$ the category of $V$-groups and $V$-homomorphisms.

\begin{examples}
\begin{enumerate}
\item If $V=\two=(\{\bot,\top\},\leq)$ with $\otimes=\wedge$, then $\two$-$\Cat$ is the category \Ord\ of preordered sets and monotone maps, while $\two$-$\Grp$ is the category \OrdGrp\ of preordered groups and monotone group homomorphisms studied in \cite{CMFM}.

\item When $V=P_+=([0,\infty],\geq)$ is the complete real half-line, with $\otimes=+$, and then $\hom(u,v)=v\ominus u=\max(v-u,0)$, $P_+$-$\Cat$ is the category \Met\ of Lawvere (generalized) metric spaces and non-expansive maps, while $P_+$-$\Grp$ is the category \MetGrp\ whose objects are the (generalized) metric groups, i.e. the Lawvere generalized metric spaces with a group operation which is a non-expansive map, and whose arrows are the non-expansive group homomorphisms.

\item If we take instead $P_\max$, so that in $([0,\infty],\geq)$ we take $\otimes=\wedge$ (note that this is $\max$ for the usual order in the real numbers), then $P_\max$-$\Cat$ is the category \UMet\ of ultrametric spaces and non-expansive maps, while $P_\max$-$\Grp$ is the category \UMetGrp\ of (generalized) ultrametric groups and non-expansive group homomorphisms.
\end{enumerate}
\end{examples}

We will be interested especially in the particular case in which $V$ is \emph{cartesian, that is, the tensor product $\otimes$ is the cartesian product $\wedge$, and $\kappa=\top$.} One of the key properties of this tensor is its idempotency, meaning $u\otimes u= u$ for all $u\in V$. In fact, its idempotency characterizes it, in the following sense.

\begin{lemma}\label{lem:idpt}
If $V$ is a commutative and unital quantale with an idempotent tensor whose unit $\kappa$ is $\top$, then $V$ is cartesian.
\end{lemma}
\begin{proof}
Assume that $\kappa=\top$ and $\otimes$ is idempotent. First we point out that, if $u,v\in V$, then $u\leq v$ if and only if $u\otimes v=u$. Indeed, if $u\leq v$, then $u=u\otimes u\leq u\otimes v\leq u$; conversely, if $u\otimes v=u$, then $u=u\otimes v\leq k\otimes v=v$.

Now, we know that, when $\kappa=\top$, $u\otimes v\leq u\wedge v$; to show the converse, we note that, if $u,v,w\in V$ are such that $w\leq u$ and $w\leq v$, then $w\otimes(u\otimes v)=(w\otimes u)\otimes v=w\otimes v=w$, and therefore $w\leq u\otimes v$.
\end{proof}

Even for non-idempotent tensor products one can consider $V$-group structures $(X,a)$ which are idempotent, in the sense that $a(x,x')\otimes a(x,x')=a(x,x')$ for all $x,x'\in X$. This is the case, for instance, of every $V$-group which is preordered (i.e. $a(x,x')$ is either $\bot$ or $\top$).

\section{Action representability for $V$-groups}

From now on, we will assume that $V$ is an integral quantale, which means that $\kappa = \top$. As observed in \cite{CM}, this implies that $\VGrp$ is a pointed category. Let
\begin{equation}\label{split}
\xymatrix{(X,a)\ar[r]^-{\langle 1,0\rangle}&(X\rtimes_\varphi Y,c)\ar@<-.5ex>[r]_-{\pi_2}&(Y,b)\ar@<-.5ex>[l]_-{\langle 0,1\rangle}}
\end{equation}
be a split extension in $\VGrp$, where we are already identifying the middle group with the semidirect product of the groups $X$ and $Y$ with respect to a group action $\varphi \colon Y \to \Aut(X)$. This implies that, for every $y\in Y$, $\varphi_y\colon (X,a)\to(X,a)$ is a $V$-functor \cite[Lemma 7.1]{CM}. We want to consider the class $S$ of points in $\VGrp$ such that, in \eqref{split}, $c$ is the product structure $a \otimes b$. The product structure does not give a split extension in $\VGrp$ for all group actions $\varphi \colon Y \to \Aut(X)$. However, those actions for which this happens have a simple characterization:

\begin{theorem}[\cite{CM}, Theorem 7.2] \label{t:tensor}
Let $X$ and $Y$ be groups, $\varphi \colon Y \to \Aut(X)$ a group action, and let $X \rtimes_{\varphi} Y$ be the semidirect product defined in $\Grp$ by $\varphi$. The following assertions are equivalent:
\begin{tfae}
\item $\xymatrix{(X,a)\ar[r]^-{\langle 1,0\rangle}&(X\rtimes_{\varphi} Y,a\otimes b)\ar@<-.5ex>[r]_-{\pi_2}&(Y,b)\ar@<-.5ex>[l]_-{\langle 0,1\rangle}}$ is a split extension in $\VGrp$;
\item the map $\overline{\varphi} \colon (X\otimes Y,a\otimes b)\to (X\otimes Y,a\otimes b)$, with $(x,y)\mapsto (\varphi_y(x),y)$, is a $V$-functor.
\end{tfae}
\end{theorem}

Given a $V$-group $(X, a)$, let
\[ \VVAut(X) = \{ f \colon X \to X \, | \, f  \text{ is an isomorphism in }  \VGrp \}. \]
$\VVAut(X)$ can be equipped with the following structure $c$, which turns it into a $V$-group:
\[ \forall \, f, g \in \VVAut(X) \quad c(f, g) = \underset{x \in X}{\bigwedge} a(f(x), g(x)). \]
In fact, this structure on $\VVAut(X)$ is inherited from the structure of the exponential $[X, X]$ in the monoidal closed category $\VCat$ (see, for instance, \cite{Law}). Hence we can conclude immediately that the evaluation map $\ev\colon X\otimes\VVAut(X)\to X$, defined by $\ev(x,f)=f(x)$, is a $V$-functor.

Theorem \ref{t:tensor} can be complemented as follows:

\begin{prop} \label{p:action rep}
Let $(X,a)$ and $(Y,b)$ be $V$-groups, $\varphi \colon Y \to \Aut(X)$ a group action, and let $X \rtimes_{\varphi} Y$ be the semidirect product defined in $\Grp$ by $\varphi$. Suppose that the $V$-category structure of $Y$ is idempotent. The following assertions are equivalent:
\begin{tfae}
\item $\xymatrix{(X,a)\ar[r]^-{\langle 1,0\rangle}&(X\rtimes_{\varphi} Y,a\otimes b)\ar@<-.5ex>[r]_-{\pi_2}&(Y,b)\ar@<-.5ex>[l]_-{\langle 0,1\rangle}}$ is a split extension in $\VGrp$;
\item the map $\overline{\varphi} \colon (X\otimes Y,a\otimes b)\to (X\otimes Y,a\otimes b)$, with $(x,y)\mapsto (\varphi_y(x),y)$, is a $V$-functor;\vspace*{2mm}
\item The group action $\varphi\colon Y\to\Aut(X)$ can be corestricted to a $V$-homomorphism $\varphi \colon Y \to \VVAut(X)$.
\end{tfae}
\end{prop}

\begin{proof}
We only need to prove the equivalence between (ii) and (iii).

(ii) $\implies$ (iii): For each $y\in Y$, $\varphi_y\colon(X,a)\to(X,a)$ is a $V$-functor, as the composite of the following $V$-functors
\[\xymatrix{X\cong X\otimes 1\ar[r]^-{1_X\otimes y}&X\otimes Y\ar[r]^{\overline{\varphi}}&X\otimes Y\ar[r]^-{\pi_1}&X},\]
hence every $\varphi_y$ is an isomorphism in $\VGrp$. Using again the fact that $\overline{\varphi}$ is a $V$-functor, and so is the following composite for each $x\in X$,
\[\xymatrix{Y\cong 1\otimes Y\ar[r]^-{x\otimes 1_Y}&X\otimes Y\ar[r]^{\overline{\varphi}}&X\otimes Y\ar[r]^-{\pi_1}&X},\]
we get that, for each $y,y'\in Y$,
$b(y,y')\leq a(\varphi_y(x),\varphi_{y'}(x))$,
which gives \[b(y,y')\leq \bigwedge_{x\in X}\,a(\varphi_y(x),\varphi_{y'}(x))=c(\varphi_y,\varphi_{y'}).\]

(iii) $\implies$ (ii): Using the fact that, for all $y$, $\varphi_y$ is a $V$-functor, we get that
\[ \forall \, x, x' \in X \quad a(x, x') \leq a(\varphi_y(x), \varphi_y(x')). \]
Moreover, since by assumption $\varphi$ is a $V$-functor, we have
\[ \forall \, y, y' \in Y \quad b(y, y') \leq \underset{z \in X}{\bigwedge} a(\varphi_y(z), \varphi_{y'}(z)). \]
Combining these two facts and using idempotency of $b$ we get
\[ a(x, x') \otimes b(y, y') = a(x, x') \otimes b(y, y') \otimes b(y, y') \leq a(\varphi_y(x), \varphi_y(x')) \otimes a(\varphi_y(x'), \varphi_{y'}(x')) \otimes b(y, y') \]
which, by transitivity of $\otimes$, is less or equal than
\[ a(\varphi_y(x), \varphi_{y'}(x')) \otimes b(y, y'). \]
\end{proof}

When $V$ is a cartesian quantale, all $V$-groups have idempotent $V$-category structures. In this particular case, the previous theorem tells us that split extensions \eqref{split} underlying the product $V$-category structure correspond to $V$-homomorphisms $\varphi \colon Y \to \VVAut(X)$. This result can be rephrased by saying that the category $\VGrp$ is action representable relatively to a good class $S$ of points. In order to state this fact in a precise way, we first need to observe that $\VGrp$ is an $S$-protomodular category:

\begin{prop} \label{p:S-protomod}
If $V$ is a cartesian quantale, the category $\VGrp$ is $S$-protomodular with respect to the class $S$ of points whose domains have the product structure.
\end{prop}

\begin{proof}
The product $V$-category structure for a cartesian quantale is clearly pullback stable and closed under finite limits, hence the class $S$ is stable under pullbacks along any morphism and closed under finite limits. Let now $\eqref{split}$ be a point in $S$. Consider the following commutative
diagram
\[ \xymatrix{ (X, a) \ar[r]^-{\langle 1, 0 \rangle} \ar[dr]_f & (X
\rtimes_{\varphi} Y, a \otimes b) & (Y, b), \ar[l]_-{\langle 0, 1 \rangle} \ar[dl]^g \\
& Z \ar[u]_m & } \] where $m$ is a monomorphism. Since
$\Grp$ is a protomodular category, $m$ is an isomorphism of groups. Hence, we can identify the $V$-group $Z$ with $(X
\rtimes_{\varphi} Y, c)$ for a suitable $V$-category structure $c$, and $m$ with the identity map, so that $c \leq a \otimes b$. Proposition $7.6$ in \cite{CM} forces then $c = a \otimes b$, since any compatible $V$-category structure on a semidirect product is bigger or equal to the product one. Hence $m$ is an isomorphism of $V$-groups. This shows that every point in $S$ is strong.
\end{proof}

\begin{remark}
In the proof of the previous proposition, we only use the fact that the quantale $V$ is cartesian to show that the class $S$ is stable under pullbacks and closed under finite limits. The proof that every point in $S$ is strong actually works for any integral quantale.
\end{remark}

Combining Propositions \ref{p:action rep} and \ref{p:S-protomod} we immediately get the following:

\begin{theorem}
If $V$ is cartesian, then the category $\VGrp$ is $S$-action representable with respect to the class $S$ of points underlying the product $V$-category structure. The representing object for the functor $\SPt(\;\; , X)$ is $\VVAut(X)$.
\end{theorem}

The previous theorem is an extension of the fact -- observed in \cite[Section 8]{CM} -- that, when $V$ is cartesian, the category $\VGrp_\sym$ of symmetric $V$-groups (i.e. $V$-groups whose $V$-category structure is symmetric) is action representable. In fact, in $\VGrp_\sym$ every point has the product $V$-category structure, so that our class $S$ coincides with the class of all points in this case.

\section{$\VGrp$ is $S$-LACC}

The argument used in \cite[Proposition 4.4]{MMS SLACC} to show that \Mon\ is $S$-LACC with respect to the class $S$ of Schreier points relies on the observation that a monoid homomorphism $\varphi \colon Y \to \End(X)$ can be equivalently seen as a functor $\Phi \colon Y \to \Mon$, where the monoid $Y$ is considered as a one-object category. A very similar fact is true for $V$-groups when $V$ is cartesian. Indeed, any $V$-group $(Y, b)$, seen as a category with a unique object $*$, has a canonical $(\VCat)$-enrichment, because $Y(*, *)=Y$ is a $V$-category. Moreover, $\VGrp$ can be $(\VCat)$-enriched as follows: given two $V$-groups $(Z, d)$ and $(W, d')$, $\VGrp((Z, d), (W, d'))$ has a natural $V$-category structure $\tilde{d}$ given by
\[ \tilde{d}(f, g) = \underset{z \in Z}{\bigwedge} d'(f(z), g(z)), \]
again, inherited from the exponential $[(Z,d),(W,d')]$ in $\VCat$.
Hence, given a $V$-homomorphism $\varphi \colon Y \to \VVAut(X)$, where $(X, a)$ and $(Y, b)$ are $V$-groups, the map
\[ Y(*, *) \to \VGrp(X, X) \]
sending $y$ to $\varphi_y$ is a $V$-functor because, by assumption, $\varphi\colon Y\to\VVAut(X)$ is a $V$-functor, and so
\[ b(y, y') \leq \underset{x \in X}{\bigwedge} a(\varphi_y(x), \varphi_{y'}(x)). \]
Thus we get a (\VCat)-enriched functor $\Phi \colon (Y, b) \to \VGrp$ sending the only object $*$ of $Y$ to $(X, a)$ and $y \in Y(*,*)$ to $\varphi_y$.

We denote by $\VGrp^Y$ the category of ($\VCat$)-enriched functors $(Y,b)\to\VGrp$ and natural transformations between them.

\begin{theorem}
If $V$ is cartesian, then the category $\VGrp$ is $S$-LACC with respect to the class $S$ of points underlying the product $V$-category structure.
\end{theorem}

\begin{proof}
Let $h \colon X \to Y$ be a morphism in $\VGrp$. Thanks to Proposition \ref{p:action rep}, we get an equivalence of categories $\SPt_Y(\VGrp) \cong \VGrp^{Y}$. Hence, showing that the change-of-base functor of the subfibration $S\text{-cod}$:
\[ h^* \colon \SPt_Y(\VGrp) \to \SPt_X(\VGrp) \]
has a right adjoint is equivalent to showing that the functor
\[ L_h \colon \VGrp^{Y} \to \VGrp^{X} \]
has a right adjoint, where $L_h=\VGrp^{h}$ is given by composition with $h$. The right adjoint $R_h$ of $L_h$ is built as follows: for any $\psi \colon X \to \VGrp$, $R_h(\psi)$ is the right Kan extension of $\psi$ along $h$, which exists since $X$ is a small category and $\VGrp$ is complete (see \cite[Corollary X.3.2]{ML98}), and belongs to $\VGrp^X$ as we show below.

Given a $\VCat$-functor $\psi\colon X\to\VGrp$, with $(Z,c):=\psi(*)$ and $\psi_x:=\psi(x)\colon Z\to Z$, let $\widetilde{\psi}:=R_h(\psi)\colon Y\to\VGrp$ be defined by
\[W:=\widetilde{\psi}(*)=\{u\colon Y\to Z\,\mbox{ map}\,;\,\psi_x(u(y))=u(h(x)+y)\mbox{ for all }x\in X,\,y\in Y\},\]
with addition defined pointwise, and the $V$-category structure defined by
\[\omega(u,v)=\bigwedge_{y\in Y}\,c(u(y),v(y)),\]
which is easily seen to be a $V$-group, and, for every $y_0\in Y$, $\widetilde{\psi}_{y_0}\colon W\to W$ is defined by $\widetilde{\psi}_{y_0}(u)(y)=u(y+y_0)$, which is a $V$-functor since
\[\omega(u,v)=\bigwedge_{y\in Y}c(u(y),v(y))=\bigwedge_{y\in Y}c(u(y+y_0),v(y+y_0))=\omega(\widetilde{\psi}_{y_0}(u),\widetilde{\psi}_{y_0}(v))\]
(because, since $Y$ is a group, $y+y_0$ covers all the elements of $Y$).

The co-unit $\varepsilon\colon L_h\,R_h \to\Id_{\VGrp^X}$ of the adjunction $L_h\dashv R_h$ is given, for every $\psi\in \VGrp^X$, by
$\varepsilon_*\colon W\to Z$ with $\varepsilon_*(u)=u(0_Y)$ (where $Z=\psi(*)$ and $W=\widetilde{\psi}(*)$ are defined as above); $\varepsilon_*$ is clearly a $V$-homomorphism, and, moreover, for each $x\in X$,
the commutativity of the following diagram is guaranteed by the definition of $W$
\[\xymatrix{W\ar[r]^{\varepsilon_*}\ar[d]_{\widetilde{\psi}_{h(x)}}&Z\ar[d]^{\psi_x}\\
W\ar[r]_{\varepsilon_*}&Z.}\]
To show the universality of $\varepsilon$, given $\varphi\in\VGrp^Y$ and $\alpha\colon L_h(\varphi)\Rightarrow\psi$ we want to define $\overline{\alpha}\colon \varphi\Rightarrow\widetilde{\psi}$ so that $\overline{\alpha}$ is the unique natural transformation satisfying $\varepsilon\cdot\overline{\alpha}_h=\alpha$: let $S:=\varphi(*)$ and $\overline{\alpha}_*\colon S\to W$ be defined by $u_s(y):=(\overline{\alpha}_*(s)(u))(y)=\alpha_*(\varphi_y(s))$. It is straightforward to verify that $\varepsilon_*\cdot (\overline{\alpha}_h)_*=\alpha_*$:
\[\varepsilon_*((\overline{\alpha}_h)_*(s))=u_s(0_Y)=\alpha_*(\varphi_0(s))=\alpha_*(s),\]
and that $\overline{\alpha}$ is the unique natural transformation satisfying this condition.
\end{proof}

Similarly to what we observed at the end of the previous section, the previous theorem extends the result of \cite[Section 8]{CM} saying that, when $V$ is cartesian, the category $\VGrp_\sym$ of symmetric $V$-groups is LACC. \\

We conclude by observing that our results concerning action representability and local algebraic cartesian closedness for $V$-groups do not seem to extend to right $V$-groups. By \emph{right $V$-group} we mean a $V$-category $(X, a)$ equipped with a group operation $+$ such that, for all $z \in X$, the map $+z \colon (X, a) \to (X, a)$ sending $x$ to $x+z$ is a $V$-functor. The reason why our results seem to fail in such a context is that they are based on Theorem \ref{t:tensor}, whose proof uses in an essential way the fact that the whole group operation $+$ is a $V$-functor. We also observe that our Proposition $3.1$ in \cite{CM} is stated improperly. Indeed, the invariance by shifting condition consists of the validity of the two following equalities:
\[ (\forall x,x',x''\in X)\;\;a(x',x'')=a(x'+x,x''+x), \]
\[ (\forall x,x',x''\in X)\;\;a(x',x'')=a(x+x',x+x''). \]

\section*{Acknowledgements}
We thank Dirk Hofmann for suggesting us the proof of Lemma \ref{lem:idpt}.\bigskip

Partial financial support by the Centre for Mathematics of the University of Coimbra (CMUC, https://doi.org/10.54499/UID/00324/2025) under the Portuguese Foundation for Science and Technology (FCT), Grants UID/00324/2025 and UID/PRR/00324/2025 is acknowledged.

The second author is member of the Gruppo Nazionale per le Strutture Algebriche, Geometriche e le loro Applicazioni (GNSAGA) dell'Istituto Nazionale di Alta Matematica ``Francesco Severi''.

This work was supported by the Shota Rustaveli National Science Foundation of Georgia (SRNSFG), through grant FR-24-9660, ``Categorical methods for the study of cohomology theory of monoid-like structures: an approach through Schreier extensions''.

\end{document}